\documentclass[11pt]{amsart}

\usepackage[hmargin=0.8in,height=8.8in]{geometry}
\usepackage{bbm}
\usepackage{amssymb,amsthm, times}
\usepackage{delarray,verbatim}
\usepackage{srcltx}
\usepackage{ifpdf}
\ifpdf
\usepackage[pdftex]{graphicx}
 \else
\usepackage[dvips]{graphicx}
 \fi

\usepackage{caption}
\usepackage{ifpdf}
\usepackage{xcolor}
\usepackage{mathrsfs}
\definecolor{webgreen}{rgb}{0,.5,0}
\definecolor{webbrown}{rgb}{.8,0,0}
\definecolor{emphcolor}{rgb}{0.95,0.95,0.95}

\usepackage{hyperref}
\hypersetup{%
          colorlinks=true,
          linkcolor=webbrown,
          filecolor=webbrown,
          citecolor=webgreen,
          breaklinks=true}
\ifpdf \hypersetup{
            pdfstartview=FitH, %%Fit, FitB, FitH
            bookmarksopen=true,
            bookmarksnumbered=true
} \else \hypersetup{dvips} \fi
\usepackage{enumitem}
\usepackage{comment}
\usepackage[normalem]{ulem}
\usepackage[colorinlistoftodos,prependcaption,textsize=small]{todonotes}
\setuptodonotes{color=blue!20} %Todoリストを作った. \todo{}とすると. 青いボックスで表示される
\usepackage{tikz}

\numberwithin{equation}{section}

\newtheorem{proposition}{Proposition}[section]
\newtheorem{corollary}{Corollary}[section]
\newtheorem{remark}{Remark}[section]
\newtheorem{lemma}{Lemma}[section]

\newtheorem{definition}{Definition}[section]
\numberwithin{remark}{section} \numberwithin{proposition}{section}
\numberwithin{corollary}{section}

\newcommand {\R}{\mathbb{R}}

\newcommand {\F}{\mathcal{F}}

\newcommand {\p}{\mathbb{P}}

\newcommand {\E}{\mathbb{E}}

\newcommand{\eps}{\varepsilon}
\newcommand{\diff}{{\rm d}}

\newcommand{\conn}{\quad\text{and}\quad}
\newcommand{\1}{\mbox{1}\hspace{-0.25em}\mbox{l}}
\newcommand{\lev}{L\'{e}vy }

\newcommand{\goes}{\rightarrow \infty}

\title{A dual characterization and algorithm of general optimal stopping}
\author[M. Egami]{Masahiko Egami}
\address[M. Egami]{Graduate School of Economics,
Kyoto University, Sakyo-Ku, Kyoto, 606-8501, Japan}
\email{egami.masahiko.8x@kyoto-u.ac.jp}
\urladdr{http://www.econ.kyoto-u.ac.jp/{\textasciitilde}egami/}
\thanks{
}
\author[T. Koike]{Tomohiro Koike}
\address[T. Koike]{Graduate School of Economics,
Kyoto University, Sakyo-Ku, Kyoto, 606-8501, Japan}
\email{koike.tomohiro.4e@kyoto-u.ac.jp}
\date{}
\begin{document}
%\tableofcontents
%\listoftodos

\begin{abstract}
We present a novel dual approach to characterizing the solution to optimal stopping problems for strong Feller processes on locally compact, separable metric spaces.  Our characterization builds upon reinterpreting a standard result in optimal stopping theory as a dual formulation, which yields computationally tractable algorithms.  Our framework is applicable to a broad class of Markov processes. First, in the one-dimensional case, while previous studies have managed to comprehensively solve only diffusion and one-sided jump Lévy processes over an infinite time horizon, our proposed method successfully extends to both infinite- and finite-horizon scenarios across all these processes, including two-sided jump Lévy processes.  In addition, for multi-dimensional cases, our framework is applicable to diffusions over an infinite horizon with certain structural features such as convex stopping regions.  To illustrate this versatility, we explicitly solve several challenging problems including a case where the optimal stopping region consists of two disjoint connected components -- a structural feature that remains unaddressed in the literature.    
\end{abstract}

\maketitle \noindent 
\small{\textbf{Key words:} 
Optimal stopping; 
Dual approach; 
Optimal stopping algorithms;
Strong Feller processes;
Jump-diffusion processes;
Multidimensional diffusion processes.\\
%\noindent JEL Classification: G32, D81, C61 \\
 }\\
%%%%%%%%%%%%%%%%%%%%%%%%%%%%%%%%%%%%%%%%%%%

\section{Introduction}\label{sec:intro}

\subsection{Problem setting and our goal}\label{sec:problem}
Let   $(\Omega, \F, \p)$ be  a probability space, where $\Omega$ is the set of all possible outcomes, and $\p$ is a probability measure defined on $\F$.
Let $X=\{X_t;\ t\geq 0\}$ be a transient strong Feller process on this probability space taking values in a locally compact separable metric space $D$.
The notations $\p_x$ and $\E_x$ denote the probability measure and the corresponding expectation operator, respectively, associated with the process starting at $x$.
Let $\zeta := \inf \{ t: \ X_t \in \Delta\}$ be the lifetime of the process, where $\Delta \notin D$ is a cemetery point adjoined to $D$.

We consider the following optimal stopping problem:
\begin{align}\label{eq:value-function}
    v(x) = \sup_{\tau \in \mathcal{T}} \E_x[g(X_\tau)]; \quad x \in D,
\end{align}
where $\mathcal{T}$ is the set of all $\mathbb{F}$-stopping times.
We assume that the reward function
\(g\) is a nonnegative function in \(\mathcal{C}_0(D)\), the space of continuous functions on \(D\) that vanish at $\Delta$. For simplicity, we write \(\mathcal{C}_0\) instead of \(\mathcal{C}_0(D)\).
In Section \ref{sec:general}, we consider the more general case in which $g$ need not vanish at $\Delta$.
Although we do not explicitly consider discounting, there is no loss of generality, since a discounted optimal stopping problem can be reduced to an undiscounted one by considering the corresponding exponentially killed process.
We also note that our framework also accommodates finite-horizon problems. In this case, we consider the Markov process \(\tilde{X}_t=(t,X_t)\) in place of \(X_t\) and the reward function \(\tilde{g}(t,x)=\mathbf{1}_{[0,T]}(t)g(x)\) in place of \(g\), where $T$ is the finite time horizon. Indeed, we solve a finite-horizon problem in Section \ref{subsec:Kou-finite}.

We define the stopping region $\Gamma$ and continuation region $C$ as follows;
\[\Gamma := \{ x \in D; \ v(x)=g(x) \} \conn C := \{ x \in D; \ v(x)>g(x) \}. \]

Our main goal in this paper is to provide a new characterization of the optimal stopping problem \eqref{eq:value-function}, yielding algorithms for its solution.
We use the following \textit{dual} formulation of optimal stopping: for any fixed $x_0 \in D$,
\begin{align}
    v(x_0)=\inf_{f\geq g, f \in \mathcal{R}} f(x_0) = \sup_{\mu \in \mathcal{S}_{x_0}} \langle g, \mu \rangle,\label{eq:duality}
\end{align}
where $\mathcal{R}$ is the space of all excessive functions with respect to $X$
, which is defined in Section \ref{sec:preliminaries},
and $\mathcal{S}_{x_0}$ is the set of laws of the random variable \(X_{\tau }\) under the probability measure \(\p_{x_0}\), taken across all stopping times \(\tau \).
Here,
\[ \langle g,\mu\rangle := \int_D g(x)\mu(dx)\]
denotes the integral of $g$ with respect to the measure $\mu$.
Indeed, \eqref{eq:duality} follows from a standard result in optimal stopping; that is, the value function is the least excessive majorant of $g$ (see \cite{Shiryaev_2008}).
However, this dual formulation derives a quantity which provides a helpful tool to solve optimal stopping, which we refer to as \textit{the duality gap} introduced in Section \ref{sec:algorithm}.

\subsection{Our motivation and contributions}
Optimal stopping problems have been extensively studied across various disciplines.  
Even in the field of infinite-horizon problems, tractable characterizations--in the sense of yielding constructive algorithms for solutions--remain limited only to one-dimensional diffusions (e.g. \cite{DK2003}) and one-dimensional one-sided jump \lev processes (e.g. \cite{egami2025potential}).  The existing studies on two-sided jumps predominantly focus on settings where the continuation region is one-sided (e.g.(\cite{surya2007approach, deligiannidis2009optimal, christensen2013optimal}), leaving problems with complex stopping regions generally unsolved.  Similarly, while notable progress has been made in multi-dimensional settings (e.g., \cite{christensen2019optimal, christensen2011harmonic, christensen2018multidimensional}),
obtaining tractable characterizations remains generally challenging.

Against this backdrop, this paper presents a novel characterization of optimal stopping problems by leveraging the dual formulation in \eqref{eq:duality}. From this formulation, we naturally define the difference between the primal and dual values associated with any candidate set $A$ contained within the  stopping region, which we refer to as the \textit{duality gap}. Our proposed characterization provides a systematic procedure for constructing a set $A$ for which the duality gap vanishes, thereby yielding the solution to the optimal stopping problem. Furthermore, this framework leads to tractable algorithms across a wide range of settings, including:
(i) Markov processes in which every singleton is regular for itself (covering both infinite- and finite-horizon problems);
(ii) Infinite-horizon multi-dimensional diffusion processes with convex stopping regions.
Concrete algorithms for setting (i) are established in Corollaries \ref{coro:naive-algorithm} and \ref{coro:finite-horizon-algorithm}, while setting (ii) is addressed in Corollary \ref{coro:convex-algorithm}. Representative problems for each setting are explicitly solved in Section \ref{sec:examples}.

A key feature of our proposed algorithms presented in Corollaries \ref{coro:naive-algorithm} and \ref{coro:finite-horizon-algorithm} is that they do not require a priori guesses regarding the geometric structure of the stopping (or continuation) region. Consequently, they remain fully applicable even when these regions exhibit non-trivial geometries, such as non-one-sided structures. Indeed, in Sections \ref{subsec:Kou-infinite} and \ref{subsec:Kou-finite}, we successfully solve problems where the stopping region consists of two disconnected components—a scenario that typically poses substantial technical challenges for traditional free-boundary methods (\cite{surya2007approach, deligiannidis2009optimal, christensen2013optimal}).

We also highlight the contrast between our approach and existing Monte Carlo-based algorithms for optimal stopping (e.g., \cite{rogers2002monte, haugh2004pricing, andersen2004primal, longstaff2001valuing}). While Monte Carlo methods similarly avoid making structural assumptions on the stopping region—typically by simulating sample paths or constructing approximate optimal martingales—they compute the value function $v(x)$ only at a specific point $x$. In contrast, our algorithm directly yields the optimal stopping policy, thereby providing a global structural characterization of the stopping region.
Furthermore, a key feature of our algorithm is that it provides an inner approximation of the stopping region: at each finite step, the computed set is guaranteed to be a subset of the true stopping region. This property is particularly useful in practice, as early termination of the algorithm still yields a valid, conservative sub-optimal stopping rule.

It is worth noting that dual formulations of optimal stopping have a long history in the literature. For example, pioneering works by \cite{bismut1977dualite, bismut1979controle} formulate optimal stopping problems as convex optimization problems using randomized stopping times and provide characterizations based on convex duality. Our framework, however, does not require such full convex-duality machinery; the simpler dual formulation in \eqref{eq:duality} suffices for our constructive algorithmic scheme.

The remainder of the paper is organized as follows. Section \ref{sec:preliminaries} provides the mathematical preliminaries used throughout the paper, with a particular focus on the fundamental properties of Markov processes. Section \ref{sec:algorithm} presents our main theoretical results.
We provide a characterization of the solution of \eqref{eq:value-function} in Section \ref{subsec:chara}.
In Section \ref{subsec:algorithm}, using this result, we develop some algorithms to solve the problems.
In Section \ref{sec:general}, we consider the more general setting in which the payoff function $g$ does not necessarily vanish at the cemetery state $\Delta$. Finally, Section \ref{sec:examples} applies the proposed algorithms to several concrete optimal stopping problems. Specifically, Sections \ref{subsec:Kou-infinite} and \ref{subsec:Kou-finite} treat the infinite- and finite-horizon optimal stopping problems for a Kou jump-diffusion process, respectively. Lastly, Section \ref{subsec:multi} addresses a multi-dimensional problem where $X$ is a two-dimensional Brownian motion and $g$ is a quadratic payoff function.

\section{Preliminaries}\label{sec:preliminaries}

In this section, we introduce the fundamental concepts and properties of Markov processes that will be used throughout the paper. 
Let $(H_t)_{t \ge 0}$ denote the Markov kernel (transition semigroup) associated with the transient process $X=\{X_t;\  t \ge 0\}$ introduced in Section \ref{sec:problem}. 
Recall that the process $X$ is assumed to be a strong Feller process. By definition, this means that for any $t > 0$, the operator $H_t$ maps $\mathcal{B}_b(D)$ into $\mathcal{C}_b(D)$, where $\mathcal{B}_b(D)$ and $\mathcal{C}_b(D)$ denote the spaces of bounded Borel measurable functions and bounded continuous functions on $D$, respectively.

Next, we recall the definition of excessive functions, which play a central role in the theory of optimal stopping. 
A non-negative measurable function $f$ on $D$ is called \textit{excessive} if it satisfies the following two conditions:
(i) $H_t f(x) \le f(x)$ for all $t \ge 0$ and $x \in D$;
(ii) $\lim_{t \downarrow 0} H_t f(x) = f(x)$ for all $x \in D$.
Under the strong Feller assumption, every excessive function is lower semi-continuous; see Exercise 2.16 in \cite{blumenthal2007markov}.

For a Borel set $A \subset D$, we introduce two stopping times: the first entry time $D_A$ and the first hitting time $T_A$. They are defined respectively as
\begin{align*}
    D_A := \inf \{ t \ge 0 : X_t \in A \}, \conn
    T_A := \inf \{ t > 0 : X_t \in A \},
\end{align*}
where we adopt the standard convention that $\inf \emptyset = \infty$. 
A point $x \in D$ is said to be \textit{regular} for the set $A$ if $\p_x(T_A = 0) = 1$. 
If the set $A$ satisfies the condition
\begin{align}\label{eq:condition-regular}
    \text{for every}\  x \in A, \ x \ \text{is regular for} \ A,
\end{align}
then the first hitting time $T_A$ and the first entry time $D_A$ coincide $\p_x$-almost surely for all $x \in D$.

For a stopping time $T$, we define the corresponding hitting operator $H_T$ acting on a measurable function $f$ by 
\[H_T f(x) := \E_x[f(X_T) \mathsf{1}_{\{T < \zeta\}}].\]
When $T = T_A$ for a Borel set $A$, we simply write $H_A$ for $H_{T_A}$.
We conclude this section by stating some useful properties of these operators.
Let $A$ and $B$ be Borel sets with  $A \subset B$ and assume that  $A$ satisfies \eqref{eq:condition-regular}.
Let $A^* :=A \cup A^r$, where $A^r:=\{x: x \ \text{is regular for} \ A\}$ .
Since $A \subset B$ and $A$ satisfies \eqref{eq:condition-regular}, we have $\p_x [T_B=0] \geq \p_x [T_A=0]=1 $ for any $x \in A^*$.
Hence, since $H_A$ is concentrated on  $A^*$ (Theorem 3.4.2 in \cite{chung-walsh}), it follows that 
\begin{align}\label{eq:HAHB}
    H_A = H_A H_B.
\end{align}
%This identity follows from the strong Markov property and the composition rule for hitting times,\(T_A=T_B+T_A\circ\theta_{T_B}\)where $\theta_t$ denotes the shift operator at time $t$ (see Proposition 3.4.1 in \cite{chung-walsh}).
Furthermore, if $f$ is an excessive function, then $H_A f$ is also an excessive function (see Proposition 3.4.4 in \cite{chung-walsh}).

\section{Main Result}\label{sec:algorithm}
First, we provide a characterization of the solution of \eqref{eq:value-function} in Section \ref{subsec:chara}.
Our main result is presented in Proposition \ref{prop:algprithm}.
In Section \ref{subsec:algorithm}, using this result, we develop some algorithms to solve the problems.

\subsection{A Characterization of Optimal Stopping with Duality}\label{subsec:chara}
To obtain our characterization, we introduce the duality gap (Definition \ref{def:duality-gap}). We first establish the following lemma.
\begin{lemma}\label{lemma:HAg-excessive}
    If non-empty $A$ is a subset of  $\Gamma$ satisfying \eqref{eq:condition-regular}, then
    $H_A g $ is excessive.
    In particular, it is lower semi-continuous.
\end{lemma}
\begin{proof}
    Since, $A$ satisfies \eqref{eq:condition-regular}, it follows from \eqref{eq:HAHB} that $H_A g = H_A H_\Gamma g = H_A v$. 
Since the right-hand side is excessive by Proposition 3.4.4 in \cite{chung-walsh}, we conclude $H_A g$ is excessive.
In particular, it is lower semi-continuous under the strong Feller assumption (see Section \ref{sec:preliminaries}).
\end{proof}
By Lemma \ref{lemma:HAg-excessive}, for $A \subset \Gamma$ satisfying \eqref{eq:condition-regular},
\begin{align}
    f_A = H_A g + \sup_{D} (g-H_A g) \label{eq:f_A}
\end{align}
is excessive and $f_A \geq g$.
 We define the duality gap of a set $A \subset \Gamma$ with \eqref{eq:condition-regular}.
\begin{definition}\label{def:duality-gap}
    Assume that non-empty $A \subset \Gamma$ satisfies \eqref{eq:condition-regular}.
    Let $x_0 \in D$ be a fixed reference point.
    Let $\mu_A^{x_0}(\diff y) := \p_{x_0} [X_{D_A} \in \diff y]$. \textit{The duality gap} $D(A)$ of $A$ is defined as 
    \begin{align}\label{eq:duality-gap}
        D(A) := f_A(x_0)-\langle g,  \mu_A^{x_0} \rangle = \sup_D (g-H_A g),
    \end{align}
    where the second equality is obtained from \eqref{eq:f_A}. Moreover, we refer to
    \begin{align}
        e_A := g-H_A g 
    \end{align}
    as \textit{the error function} for $A$. Note that $D(A) = \sup_D e_A$.
\end{definition}
\begin{remark}
\begin{enumerate}
\item Note that \(D(A)\) is independent of the choice of the reference point \(x_0\). Henceforth, we do not explicitly write the reference point.
\item The quantity \(D(A)\) measures the maximum payoff loss from waiting until the process reaches \(A\), rather than stopping immediately, and thus indicates how far \(A\) is from the stopping region. This interpretation is confirmed by Lemma \ref{lemma:gap-zero}.

    \item Since $g=H_A g$ on $A$, $D(A)$ is always non-negative.
\end{enumerate}
\end{remark}

A key feature of our characterization is that, rather than directly characterizing the stopping region \(\Gamma\), it characterizes a set that is \(g\)-equivalent to \(\Gamma\), as defined below. 

\begin{definition}
   A set $A \subset D$ satisfying \eqref{eq:condition-regular} is said to be \textit{$g$-equivalent} to $\Gamma$ if $H_A g = H_\Gamma g=v$.
\end{definition}
The second equality $H_\Gamma g =v$ follows from Theorem 3.3.3 in \cite{Shiryaev_2008}.
The following lemma shows that solving the optimal stopping problem reduces to finding a set for which the duality gap is zero.

\begin{lemma}\label{lemma:gap-zero}
  Assume that a non-empty set $A \subset \Gamma$ satisfies the condition \eqref{eq:condition-regular}.
  If $D(A) = 0$, then we have
  \[v(x)= \langle g,  \mu_A^{x} \rangle = H_{A} g(x); \quad x \in D. \]
  Hence,  $A$ is $g$- equivalent to $\Gamma$.
\end{lemma}
\begin{proof}
Assume that $D(A)=0$. We have $f_A(x)=\langle g,  \mu_A^{x} \rangle$.
It follows from \eqref{eq:duality} that
\[ v(x)= \inf_{f\geq g, f \in \mathcal{R}} f(x)  \leq f_A(x)=\langle g,  \mu_A^{x} \rangle \leq \sup_{\mu \in \mathcal{S}_x} \langle g, \mu \rangle = v(x).\]
Therefore, we conclude $v(x)= \langle g,  \mu_A^{x} \rangle$.

\end{proof}

We present our main result. 
\begin{proposition}\label{prop:algprithm}
Let $\{A_n\}_n$ be an increasing sequence such that each $A_n$ satisfies \eqref{eq:condition-regular}. 
Let $A_\infty = \lim_{n \goes} A_n= \bigcup_n A_n$. Assume the following assumptions:
\begin{enumerate}[label=(\roman*)]
    \item $A_n \subset \Gamma$ for each $n$,
    \item $A_{n+1} \supset A_n \cup \arg\max  e_n $ for each $n$,
    \item The closure $\overline{A_\infty}$ satisfies \eqref{eq:condition-regular}.
\end{enumerate}
Then, we have $D(A_\infty) = 0$.
Hence, $v=H_{A_\infty} g$.
\end{proposition}

An important feature of Proposition \ref{prop:algprithm} is that it constructs a set $A_\infty$, which is \(g\)-equivalent to \(\Gamma\), rather than identifying the entire stopping region \(\Gamma\). As we will see in Corollary \ref{coro:naive-algorithm}, $A_\infty$ can even be countable. This reduced representation plays a crucial role in developing a practical algorithm, as it allows us to construct the value function through only countably many operations, without having to identify the entire stopping region.
We prepare  several lemmas to prove this proposition.

\begin{lemma}\label{lemma:max-in-Gamma}
    Assume that a non-empty set $A \subset \Gamma$ satisfies the condition \eqref{eq:condition-regular}.
   Then, $e_A=g-H_A g$ attains its maximum on $\Gamma$.
\end{lemma}
\begin{proof}
    Since $g$ vanishes at $\Delta$ and $H_A g \geq 0$,
    $e_A$ vanishes at $\Delta$.
    Moreover, $e_A$ is upper semi-continuous by Lemma \ref{lemma:HAg-excessive} and $e_A =0$ on $A$.
   It follows that \(e_A\) attains its maximum on \(D\).
    
    Let $x_A$ be a point where $e_A$ attains its maximum.
    %Assume to the contrary that $x_A \in D \setminus \Gamma = C$.
    %By this assumption and
    By the definition of $e_A$, we have, for each $x \in \Gamma$,
    \begin{align}
        g(x_A)-H_A g (x_A) \geq  g(x)-H_A g (x) \label{eq:xa-continuation-region}.
    \end{align}
    We have 
    \begin{align}
        H_A g(x_A) 
        & = \int H_A g(x) \p_{x_A}[X_{T_\Gamma} \in \diff x] \label{eq:HAg-representation} \\
        &\geq \int (g(x)-g(x_A)+H_A g(x_A)) \p_{x_A}[X_{T_\Gamma} \in \diff x] \nonumber \\
        &=\int g(x) \p_{x_A}[X_{T_\Gamma} \in \diff x] -\p_{x_A}[X_{T_\Gamma}\in D]\big(g(x_A)-H_A g(x_A)\big) \nonumber\\
        &\geq \int g(x) \p_{x_A}[X_{T_\Gamma} \in \diff x] -g(x_A)+H_A g(x_A) \nonumber\\
        & = H_{\Gamma} g(x_A) -g(x_A)+H_A g(x_A)\nonumber, 
    \end{align}
    where the first equality follows from $A \subset \Gamma$ and the strong Markov property, the first inequality follows from \eqref{eq:xa-continuation-region} and the second inequality follows from $D(A)= g(x_A)-H_A g(x_A) \geq 0$.
    \eqref{eq:HAg-representation} implies that  $v(x_A) = H_{\Gamma} g(x_A) \leq g(x_A)$, where we use Theorem 3.3.3 in \cite{Shiryaev_2008} to obtain the first equality.
    Therefore, $x_A \in \Gamma$.
\end{proof}

\begin{lemma}\label{lemma:increasing-pointwise}
    Let $\{A_n\}_n$ be a sequence satisfying the assumptions in Proposition \ref{prop:algprithm}.
    Then, $H_{A_n} g$ increasingly converges to $H_{A_\infty} g$ pointwise.
\end{lemma}
\begin{proof}
    For each $n$, since $H_{A_{n+1}} g$ is excessive by Lemma \ref{lemma:HAg-excessive}, we have
    \[H_{A_{n+1}} g \geq H_{A_n} H_{A_{n+1}} g = H_{A_n} g, \]
    where the second equality follows from \eqref{eq:HAHB}.
    Hence, $H_{A_n}g$ is increasing in $n$.
    Fix $x \in D$. 
    Since $A_n \uparrow A_\infty$, $D_{A_n} \downarrow D_{A_\infty}$ almost surely. 
    Moreover, by the dominated convergence theorem, the continuity of $g$ and  the right-continuity of the paths, it follows that
    \begin{align*}
        \lim_{n \goes} H_{A_n}g(x)
        &=\lim_{n \goes} \E_x[g(X_{D_{A_n}})] 
        = \E_x[\lim_{n \goes }g(X_{D_{A_n}})]\\
        &= \E_x[g(\lim_{n \goes } X_{D_{A_n}})] =
        \E_x[g( X_{D_{A_\infty}})] = H_{A_\infty} g(x).
    \end{align*}
\end{proof}
We give a proof of Proposition \ref{prop:algprithm}.
\begin{proof}[Proof of Proposition \ref{prop:algprithm}]
   The second statement follows from the first statement and Lemma \ref{lemma:gap-zero}.
   It remains to show the first statement.
   Recall that $D(A_n) = \sup_D (g- H_{A_n} g)$.
   Since  $\lim_{n \goes} D(A_n) \geq D(A_\infty) \geq 0$ by Lemma \ref{lemma:increasing-pointwise},
   it suffices to show that $\lim_{n \goes} D(A_n)=0.$

   We define, for each $n$,
   \begin{align}
       f_n = \1_{\overline{A_\infty}} (g-H_{A_n}g).
   \end{align}

   Since $\arg \max e_n \subset A_\infty$ for each $n$ by the assumption (ii), we have
   \[\lim_{n \goes} D(A_n) = \lim_{n \goes} \sup_D f_n. \]
   Hence, we show that the right-hand side is zero.
   
   It follows from Lemma \ref{lemma:increasing-pointwise} that $f_n \downarrow f_\infty$ pointwise, where 
   \[ f_\infty := \1_{\overline{A_\infty}} (g-H_{A_\infty}g),\]
   which is identically zero.
   Indeed, for $x \in D \setminus \partial A_\infty$, we clearly have $f_\infty(x)=0$. For $x \in  \partial A_\infty$, $f_\infty(x)=0$  follows from the assumption (iii).
   Hence, $f_n \downarrow 0$ pointwise. In particular, $f_n$ is non-negative.
   Since  $\overline{A_\infty}$ is closed,  $f_n$ is upper semi-continuous by Lemma \ref{lemma:HAg-excessive} .

   Since \(f_n\) and \(H_{A_n}g\) are non-negative, we have
\begin{align*}
0 \leq f_n(x)
&\leq \max\{g(x)-H_{A_n}g(x),0\} \leq \max\{g(x),0\}, \qquad x\in D.
\end{align*}
Since \(g\) vanishes at \(\Delta\), it follows that \(f_n\) also vanishes at \(\Delta\).
   Thus, \(f_n\) extends to an upper semi-continuous function on the one-point compactification \(D^*=D\cup\{\Delta\}\).
   Therefore, Dini's theorem yields \(f_n \downarrow 0\) uniformly, which implies that $\lim_{n \goes} \sup_D f_n=0$.
   \end{proof}

  \subsection{Algorithms}\label{subsec:algorithm}
  Using the result in Proposition \ref{prop:algprithm}, we develop algorithms for some problems.
  They include (i) Markov processes in which every singleton satisfies \eqref{eq:condition-regular} (covering both infinite- and finite-horizon problems); (ii) Infinite-horizon multi-dimensional diffusion processes with convex stopping regions. 
  
   When every singleton satisfies \eqref{eq:condition-regular}, we can use the following simple algorithm. This condition is satisfied by a broad class of one-dimensional stochastic processes, including jump-diffusion processes with a non-degenerate Brownian component, as well as Brownian motion itself. It can also hold in the absence of a Brownian component; for example, every one-dimensional \lev process of unbounded variation satisfies this condition.
   For a more detailed characterization for \lev process, see Theorem 6.5 of \cite{Kyprianou_2014}.
   
   \begin{corollary}\label{coro:naive-algorithm}
       Assume that every singleton satisfies \eqref{eq:condition-regular}. We recursively define
\[
A_{n+1}=A_n\cup\arg\max e_n,\qquad n\geq1,
\]
with \(A_1=\arg\max g\). Then,
\(A_\infty=\lim_{n\to\infty}A_n
\) satisfies \(D(A_\infty)=0\).
   \end{corollary}
   
\begin{proof}
It follows from Lemma \ref{lemma:max-in-Gamma} that Assumption (i) in Proposition \ref{prop:algprithm} is satisfied. Assumption (ii) is immediate from the construction. It remains to verify that Assumption (iii) is satisfied.
    Let $A \subset \Gamma$ be an arbitrary set. For every point $x \in A$, we have $T_A \leq T_{\{x\}}$ almost surely. By the assumption, $\p_x[T_{\{x\}}=0]=1$. Thus, $\p_x[T_{A}=0]=1$ as well.
    Hence, an arbitrary set satisfies \eqref{eq:condition-regular}.
    Therefore,     
    the assumption (iii) in Proposition \ref{prop:algprithm} is satisfied. 
\end{proof}

We consider a finite-horizon optimal stopping problem:

$$
v(t,x)=\sup_{\tau\in[t,T]}\mathbb{E}_{t,x}[g(X_\tau)],
\qquad t\in[0,T].
$$

By defining the Markov process \(Y_t=(t,X_t)\), we can reformulate this problem as an optimal stopping problem for $Y$:

$$
v(t,x)=\sup_{\tau}\E_y[h(Y_\tau)],
$$
where
$$
h(y)=h(t,x)=g(x)\mathbf{1}_{[0,T]}(t).
$$
In the following corollary, the error function \(e_n\) and the duality gap \(D\) are defined with respect to the process \(Y\), rather than \(X\).

\begin{corollary}\label{coro:finite-horizon-algorithm}
Assume that every singleton satisfies \eqref{eq:condition-regular} with respect to the process $X$.
Define
\[     A_1
        := \{(t,x): t\in[0,T],\,x\in\Gamma_\infty\}
        \cup \{t\geq T\},
   \]
where $\Gamma_\infty$ denotes the stopping region of the infinite-horizon problem.
Recursively define
\begin{align*}
A_{n+1} &:= A_n\cup\gamma_n,\\
\gamma_n
&:= \bigcup_{(t_n,x_n)\in\arg\max e_n}
\{(t,x_n):t\geq t_n\}, \quad n \geq 1.
\end{align*}
Then, the limiting set
\(        A_\infty:=\lim_{n\to\infty}A_n
   \)
satisfies
\(        D(A_\infty)=0.
   \)
\end{corollary}

\begin{proof}
Let $\Gamma \subset [0,T]\times D$ be the stopping region, and define
\[  \Gamma(t):=\{x\in D:(t,x)\in\Gamma\}.
   \]
First, we prove that \(   \Gamma(t_n)\subset\Gamma(t)\)
for any $t \in [t_n,T]$ and $n \geq 0$.
 Since
$\{\tau:\tau\in[t,T]\}\subseteq\{\tau:\tau\in[t_n,T]\},$
we have 
$
v(t_n,x)\geq v(t,x).
$
Now, if $x\in\Gamma(t_n)$, then
$v(t_n,x)=g(x).$
Since $v(t,x)\geq g(x)$, it follows that
$$
g(x)=v(t_n,x)\geq v(t,x)\geq g(x),
$$
and hence $v(t,x)=g(x)$. Therefore, $x\in\Gamma(t)$, proving
$$
\Gamma(t_n)\subseteq\Gamma(t), \qquad t \in [t_n, T].
$$
Since $(t_n,x_n)\in\Gamma$ by Lemma \ref{lemma:max-in-Gamma}, we have
$\gamma_n\subset\Gamma$. Hence, assumption (i) in Proposition \ref{prop:algprithm} is satisfied.

It remains to verify that each $A_n$ and $A_\infty$ satisfy
\eqref{eq:condition-regular}.
Let $(t,x)\in A_n$. By construction,
\(   (s,x)\in A_n\) for all \(s\geq t\).
In particular, $B:= \{ (x,s): s \geq t \} \subset A_n \cap (\R \times [t,\infty)) $.
Since the singleton $\{x\}$ satisfies \eqref{eq:condition-regular}, \[ \p_{(t,x)}[Y_{T_{A_n}}=0] \geq \p_{(t,x)}[Y_{T_B}=0] = \p_{t,x}[X_{T_{\{x\}}}=0] =1. \]
Hence, each $A_n$ satisfies \eqref{eq:condition-regular}.
The same argument applies to $A_\infty$.
\end{proof}

In multidimensional Markov processes, singleton sets generally do not satisfy \eqref{eq:condition-regular} (e.g., a singleton is polar for a $d$-dimensional Brownian motion when $d \ge 2$). Therefore, Corollary \ref{coro:naive-algorithm} cannot be directly applied in the multidimensional setting.
Nevertheless, an algorithm can still be constructed if additional structural conditions are present. For instance, suppose that $X$ is a multidimensional diffusion process, $\Gamma$ is a convex set, and there exists a set $A_1 \subset \Gamma$ satisfying \eqref{eq:condition-regular}. In this case, one can construct an algorithm of the following form.

\begin{corollary}\label{coro:convex-algorithm}
  Assume that \(X\) is a transient diffusion process in a domain \(D\subset\mathbb{R}^d\), where \(d\geq2\), satisfying Assumption A in \cite{pinsky1995}. Assume that \(\Gamma\) is convex and \(A_1\subset\Gamma\). Define recursively
\[
A_{n+1}=\operatorname{conv}\bigl(A_n\cup\arg\max e_n\bigr),\qquad n\geq1,
\]
and $A_\infty = \lim_{n\goes} A_n$.
Then, $D(A_\infty) = 0$.
\end{corollary}
\begin{proof}
   Since $\Gamma$ is convex and $A_1\cup\arg\max e_1\subset\Gamma$ by Lemma \ref{lemma:max-in-Gamma}, we have
$A_2=\operatorname{conv}\bigl(A_1\cup\arg\max e_1\bigr)\subset\Gamma.$
The same argument shows that $A_n\subset\Gamma$ for every $n$, and hence assumption (i) in Proposition \ref{prop:algprithm} is satisfied. Assumption (ii) is immediate.
   It remains to show that $A_n$ for $n \geq 2$ and $\overline{A_\infty}$ satisfy \eqref{eq:condition-regular}.
   It suffices to show that any convex set $S$ with nonempty interior satisfies \eqref{eq:condition-regular} since $A_n$ for $n \geq 2$ and $\overline{A_\infty}$ are convex. 
   Let $a\in S$. If $a\in\operatorname{int} S$, then it is clear that $a$ is regular for $S$. Assume that $a\in\partial S$. If $\mathbb{R}^d\setminus S$ satisfies an exterior cone condition at $a$, then $a$ is regular for $S$ by Theorem 2.3.3 (iii) in \cite{pinsky1995}.
Since $\operatorname{int} S$ is nonempty, we can take $y\in\operatorname{int} S$ and $\eps>0$ such that the ball $B(y,\eps)$ centered at $y$ is contained in $\operatorname{int} S$. Since $S$ is convex,

$$
S_1=\bigcup_{z\in B(y,\eps)}[a,z]\subset S,
$$

where $[a,z]$ denotes the line segment joining $a$ and $z$. Since $S_1$ is a cone with vertex $a$, it follows that $\R^d\setminus S$ satisfies an exterior cone condition at $a$.

\end{proof}

It should be noted that the algorithms developed in this section are by no means exhaustive. In particular, if additional qualitative properties of the  stopping region are known a priori, one can devise a more tailored and efficient algorithm based on those structural insights. For instance, while Corollary \ref{coro:naive-algorithm} presents the most naive algorithm, if the stopping region is known in advance to be one-sided, e.g., $\Gamma = (-\infty, a]$ for some $a$, a more efficient scheme can be constructed by setting $A_{n+1} = A_n \cup (-\infty, x_n)$, where $x_n = \max \{ \operatorname{argmax} e_n \}$. Thus, investigating the qualitative properties of the optimal stopping problem at hand is crucial for applying our methodology, as a deeper structural understanding enables the design of more effective algorithms. This highlights the flexibility of our approach, and developing such refined, problem-specific algorithms remains an important direction for future research.

\section{General Case}\label{sec:general}
So far, we have assumed that \(g\) vanishes at \(\Delta\). 
Indeed, this assumption ensures that \(e_A\) attains its maximum on \(D\), which is essential for our algorithm.
However, for a general continuous, nonnegative \(g\), there exists a sequence \(\{g_n\}\) vanishing at \(\Delta\) such that the corresponding value functions, continuation regions, and stopping regions approximate those for \(g\) in the following sense.

\begin{proposition}
Let $g$ be a non-negative continuous function on $D$.
Let $(g_n )_n \subset \mathcal{C}_0$ be a sequence with the following conditions:
\begin{enumerate}[label=(\roman*)]
    \item $g_n \uparrow g$.
    \item For each $x$, there exists $N$ such that $g_n(x)=g(x)$ if $n \geq N$.
\end{enumerate}
Let $v_n$, $C_n$ and $\Gamma_n$ be the value function, continuation region and stopping region corresponding to  $g_n$.
Then, $v_n \uparrow v$ pointwise, and furthermore, 
\[
C = \limsup_{n \to \infty} C_n = \liminf_{n \to \infty} C_n \conn \Gamma = \limsup_{n \to \infty} \Gamma_n = \liminf_{n \to \infty} \Gamma_n. 
\]
\end{proposition}

\begin{remark}
     Indeed, a sequence $(g_n)$ satisfying (i) and (ii) can be constructed as follows. Let $K_0 \subset D$ be a compact set, and for each $n \ge 1$, define
     $$ K_n = \{ x \in D : \operatorname{dist}(x, K_{n-1}) \le 1 \}, $$
     where $\operatorname{dist}(\cdot, \cdot)$ is the metric on $D$. For $n \ge 1$, define $g_n : D \to \mathbb{R}$ by
     $$ g_n(x) =  
     \begin{cases}  g(x) & \text{if } x \in K_{n-1}, \\ 
     (1 - \operatorname{dist}(x, K_{n-1})) g(x) & \text{if } x \in K_n \setminus K_{n-1}, \\  
     0 & \text{otherwise.}  
     \end{cases}$$
     Then, $(g_n)$ satisfies the required properties (i) and (ii).
\end{remark}

\begin{proof}
First, we show $v_n \uparrow v$.
Since $g_n \le g_{n+1}$, it is clear that $v_n \le v_{n+1}$. Thus, $v_n$ is increasing with respect to $n$.
Let $\hat{v} = \lim_{n \to \infty} v_n$.
We show that $\hat{v}=v$.
Since $v_n \leq v$ for each $n$, $\hat{v} \leq v$ is obtained by letting $n \to \infty$.
We show  $\hat{v} \geq v$.
Since $v_n$ is excessive for each $n$, its increasing limit $\hat{v}$ is also excessive.
Therefore, it suffices to show $\hat{v} \ge g$ since $v$ is the least excessive majorant of $g$.
Taking $n \to \infty$ in $v_n \ge g_n$, we obtain $\hat{v} \ge g$.

Next, we show $C = \limsup_{n \to \infty} C_n = \liminf_{n \to \infty} C_n $.
For this, it suffices to show:
\[
\limsup_{n \to \infty} C_n \subset C \quad \text{and} \quad C \subset \liminf_{n \to \infty} C_n.
\]
First, we show $\limsup_{n \to \infty} C_n \subset C$.
Let $x \in \Gamma$, which means $g(x) = v(x)$.
Choose $N$ such that $g_n(x)=g(x)$ if $n \geq N$.
For any $n \ge N$, $g_n(x) = g(x)$.
On the other hand, since $g_n \le g$ and $v_n \le v$, we have:
\[
g_n(x) = g(x) = v(x) \ge v_n(x) \ge g_n(x).
\]
Hence, $v_n(x) = g_n(x)$ for all $n \ge N$.
Therefore, $x \in \Gamma_n$ for $n \ge N$, which implies $x \in \bigcap_{n \ge N} \Gamma_n \subset \liminf_{n \to \infty}\Gamma_n$.
Taking the complement, since $\Gamma_n^c = C_n$, we obtain:
\[
\limsup_{n \to \infty} C_n \subset C.
\]

Next, we show $C \subset \liminf_{n \to \infty} C_n$.
Let $x \in C$.
Since $g(x) < v(x)$ and $v_n(x) \uparrow v(x)$, there exists $N$ such that:
\[
g(x) < v_n(x),\quad n \ge N.
\]
Since $g_n \le g$, we have $g_n(x) \le g(x) < v_n(x)$ for all $n \ge N$.
Thus, $g_n(x) < v_n(x)$ holds for all $n \ge N$, which means $x \in C_n$ for all $n \ge N$.
Therefore,
\[
x \in \liminf_{n \to \infty} C_n.
\]
Hence, $C \subset \liminf_{n \to \infty} C_n$.
Finally,  the relation $\Gamma = \limsup_{n \to \infty} \Gamma_n = \liminf_{n \to \infty} \Gamma_n$ follows from 
these results, together with  $\Gamma = D \setminus C$ and $\Gamma_n = D \setminus C_n$ for each $n$.
\end{proof}

\section{Examples}\label{sec:examples}
In this section, we mainly focus on problems that pose substantial challenges for existing methods in the literature. As noted earlier, even in infinite-horizon settings, one-dimensional \lev processes with two-sided jumps are generally manageable under conventional approaches only when the stopping region is known a priori to have a simple one-sided structure.
Similarly, multi-dimensional diffusions also pose substantial challenges for standard methods.
A primary reason for this limitation is that, unlike in the case of \textit{one-sided} jump processes or \textit{one-dimensional} diffusions, obtaining an explicit representation of the harmonic function, commonly referred to as the scale function, is notoriously difficult for these processes.  
The finite-horizon case presents even greater difficulty. In contrast, our algorithm does not rely on analytic expressions of harmonic functions, allowing us to tackle considerably more general problems. While we illustrate our approach using the Kou process, a two-sided exponential jump diffusion in Sections \ref{subsec:Kou-infinite} and \ref{subsec:Kou-finite}, our method does not depend on the specific choice of processes (See Corollary \ref{coro:naive-algorithm}). 
Section  \ref{subsec:multi} addresses two-dimensional diffusion processes--a class of problems for which tractable solutions are even scarcer in the literature. Specifically, we consider $X$ to be a two-dimensional Brownian motion and select a payoff function such that the corresponding stopping region is convex, leveraging Corollary \ref{coro:convex-algorithm}. To establish the convexity of the stopping region in our setting, we partly adapt the proof strategy of Proposition 4.1 in \cite{christensen2019optimal}.

In Section \ref{subsec:Kou-infinite}, we consider the infinite-horizon problem, while in Section \ref{subsec:Kou-finite}, we consider the finite-horizon problem.

\subsection{Infinite-Horizon Optimal Stopping}\label{subsec:Kou-infinite}
We consider a one-dimensional Kou jump-diffusion (\cite{kou})

$$
    dX_t=\mu\,dt+\sigma\,dB_t+dJ_t,
$$
where \(B=(B_t)_{t\geq0}\) is a standard Brownian motion, \(\sigma>0\), and \(J\) is a compound Poisson process with constant jump intensity \(\lambda>0\). The jump sizes are assumed to have a symmetric double-exponential distribution,

$$
    f_Y(y)
    =
    \frac{\eta}{2}e^{-\eta |y|},
    \qquad y\in\mathbb R,
$$
where \(\eta>0\). 
We consider a reward function consisting of two Gaussian peaks,
\[
    g(x)
    =
    \exp\left(
        -\frac{(x-a)^2}{2s^2}
    \right)
    +
    p\exp\left(
        -\frac{(x+a)^2}{2s^2}
    \right),
\]
where $p>0$ is a constant, \(a>0\) determines the locations of the two peaks and \(s>0\) controls their width. 
We assume that \(a>s\), so that the reward function \(g\) has two distinct local maxima.
We solve this problem using the algorithm in Corollary \ref{coro:naive-algorithm}. In the numerical experiments, we set
\[(\mu,\sigma,\lambda,\eta,p,a,s)=(1,0.8,0.2,2,2,3,0.6).\]
Figure \ref{fig:Dual_Kou}  visually defines the optimal stopping problem by plotting the reward function $g$, the computed value function $v$, and the computed $A_\infty$  together. The underlying reward function $g$, characterized by two distinct peaks, is represented by the black dashed line. The solid blue line traces the value function $v$. The set $A_\infty$ is plotted on this blue line as red dots.
Furthermore, Figure \ref{fig:DualityGap}  illustrates the convergence of the duality gap $D_n=D(A_n)$ as the iteration counter $n$ increases. 
The magnitude of the gap \(D_n\), shown by the solid blue line with circular markers, exhibits convergence.
%The error eventually flattens out asymptotically as the algorithm proceeds beyond $70$ iterations.

\begin{figure}
    \centering
    \includegraphics[width=0.5\linewidth]{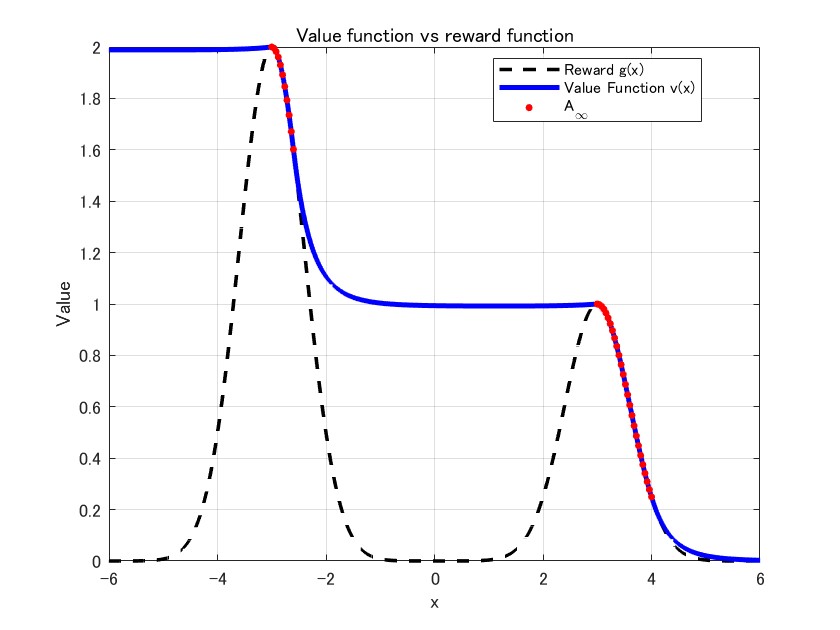}
     \captionsetup{font=scriptsize}
    \caption{Optimal stopping problem visualization: The reward function $g(x)$ (black dashed line), the computed value function $v(x)$ (solid blue line), and the set $A_\infty$ highlighted by the red dotted segments.}
    \label{fig:Dual_Kou}
\end{figure}
\begin{figure}
    \centering
    \includegraphics[width=0.5\linewidth]{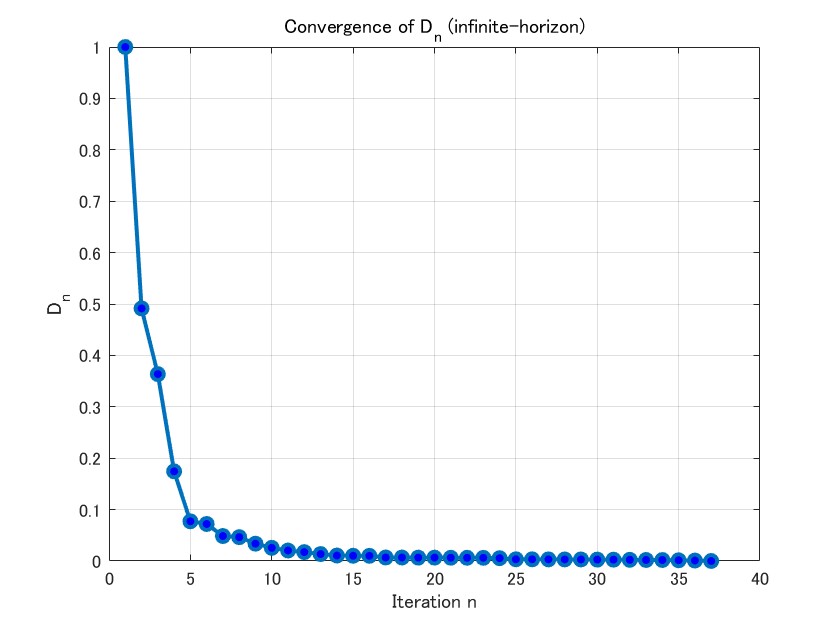}
     \captionsetup{font=scriptsize}
    \caption{Convergence of the duality gap $D_n = D(A_n)$ to zero as a function of the iteration counter $n$.}
    \label{fig:DualityGap}
\end{figure}

\subsection{Finite-Horizon Optimal Stopping}\label{subsec:Kou-finite}
 We solve the finite-horizon optimal stopping problem with maturity \(T=1\). Except for the finite maturity, all the other model parameters and settings are the same as in Section \ref{subsec:Kou-infinite}. The value functions \(x\mapsto v(t,x)\) at \(t=0,0.25,0.5,0.75,\) and \(1\) are shown in Figure \ref{fig:Dual_KouF}. The thick red line represents the value function obtained in Section \ref{subsec:Kou-infinite} for the infinite-horizon problem. As the maturity approaches, the value function decreases, while the set $A_\infty$ becomes larger, as can be seen from the figure.
Figure \ref{fig:DualityGapF} shows the convergence of the duality gap. 
%After approximately 500 iterations, the duality gap is of order \(10^{-5}\).

\begin{figure}
    \centering
    \includegraphics[width=0.6\linewidth]{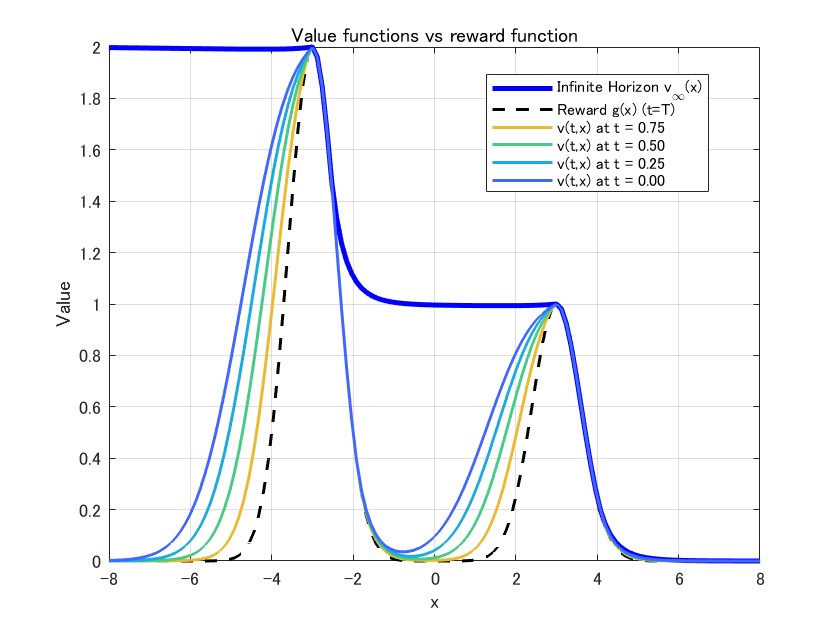}
     \captionsetup{font=scriptsize}
    \caption{ Comparison of the finite-horizon value functions $v(t,x)$ at various time slices ($t = 0.0, 0.25, 0.50, 0.75$), the reward function $g(x)$ (dashed black line), and the infinite-horizon stationary value function $v_\infty(x)$ (thick red line).}
    \label{fig:Dual_KouF}
\end{figure}
\begin{figure}
    \centering
    \includegraphics[width=0.5\linewidth]{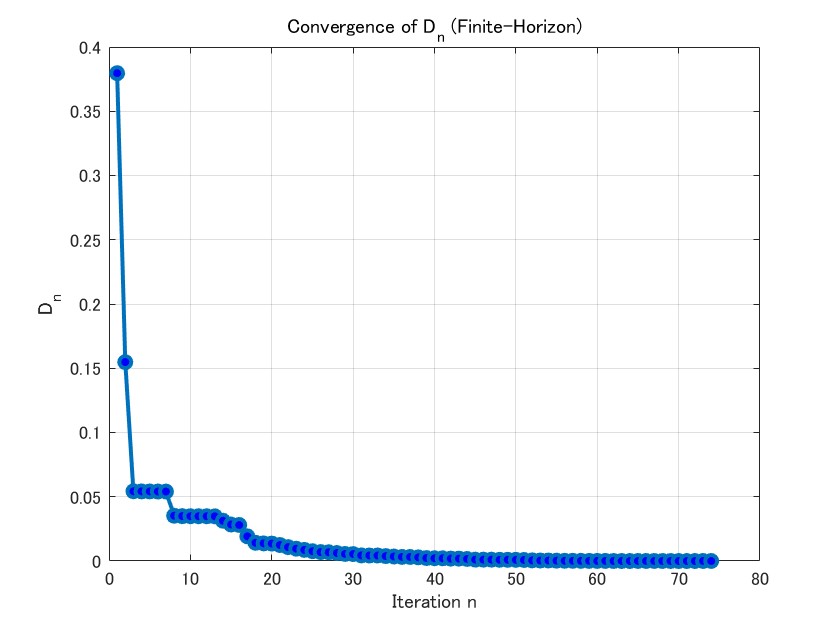}
     \captionsetup{font=scriptsize}
    \caption{Convergence of the duality gap $D_n = D(A_n)$ to zero as a function of the iteration counter $n$.}
    \label{fig:DualityGapF}
\end{figure}

\subsection{Multi-dimensional case}\label{subsec:multi}
We consider an optimal stopping problem driven by a two-dimensional Brownian motion. Let $X=(W^{(1)}, W^{(2)})$,
where $W^{(1)}$ and $W^{(2)}$ are one-dimensional Brownian motions with correlation coefficient $\rho \in [-1,1]$.
For a discount rate \(q>0\), we consider the optimal stopping problem
$$
v(x,y)
:=
\sup_{\tau\in\mathcal{T}}
\mathbb{E}_{(x,y)}
\left[
e^{-q\tau}g(X_\tau)
\right],
$$
where the reward function is given by
$$
g(x,y):=(K-ax^2-by^2)_+,
\qquad a,b,K>0.
$$

Since the stopping region of this problem is convex, which is shown in Proposition \ref{prop:convex-2D} in Appendix, we apply Corollary \ref{coro:convex-algorithm}.
In the numerical experiments, we set
\[(q, K, a,b,\rho)=(0.1, 1, 1,2,0.1).\]
Figures \ref{fig:StopRegion_2D}, \ref{fig:Value_2D}, and \ref{fig:DualityGap_2D} show the results of the numerical analysis.

Figure~\ref{fig:StopRegion_2D} illustrates the set $A_\infty$ obtained. 
As shown in the plot, the constructed stopping boundary exhibits a tilted elliptical shape.
The red dots indicate the points that maximize the error function at each iteration, i.e.,
$x_n=\arg\max e_n.$
Figure~\ref{fig:Value_2D} shows the three-dimensional profile of the computed value function $v$ alongside reward function $g$. 
Figure~\ref{fig:DualityGap_2D} depicts the historical evolution of the duality gap $D_n$ across the iterations. 
%The proposed algorithm converged after 87 iterations, with the duality gap falling below the threshold of \(10^{-5}\).

\begin{figure}
    \centering
    \includegraphics[width=0.5\linewidth]{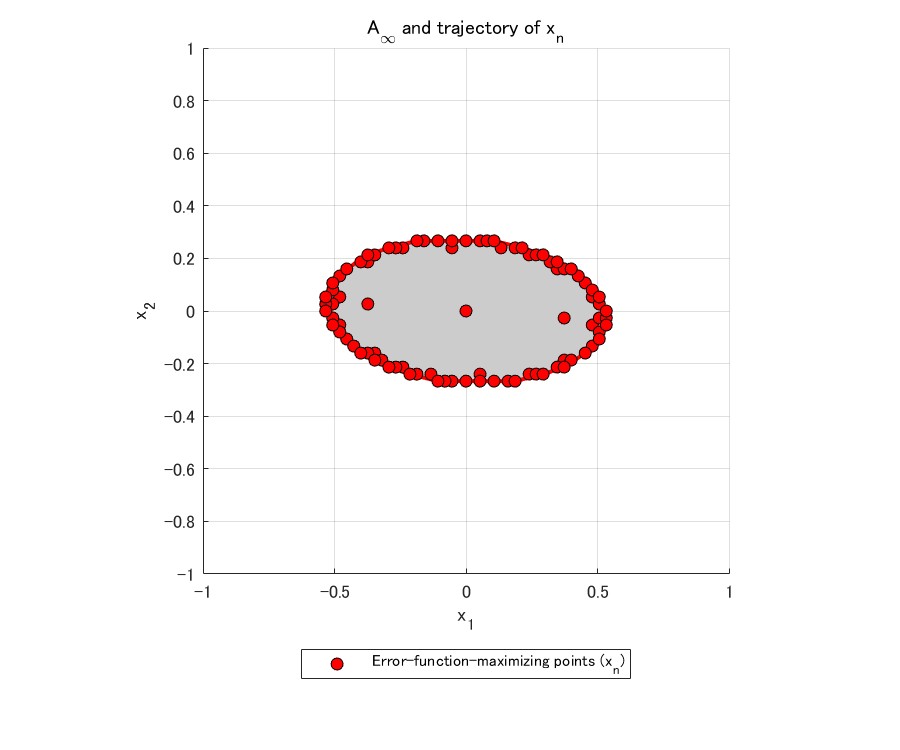}
     \captionsetup{font=scriptsize}
    \caption{The constructed $A_\infty$ for correlation coefficient $\rho = 0.10$, along with the trajectory of error-maximizing points $x_n$ (red circles).}
    \label{fig:StopRegion_2D}
\end{figure}
\begin{figure}
    \centering
    \includegraphics[width=0.5\linewidth]{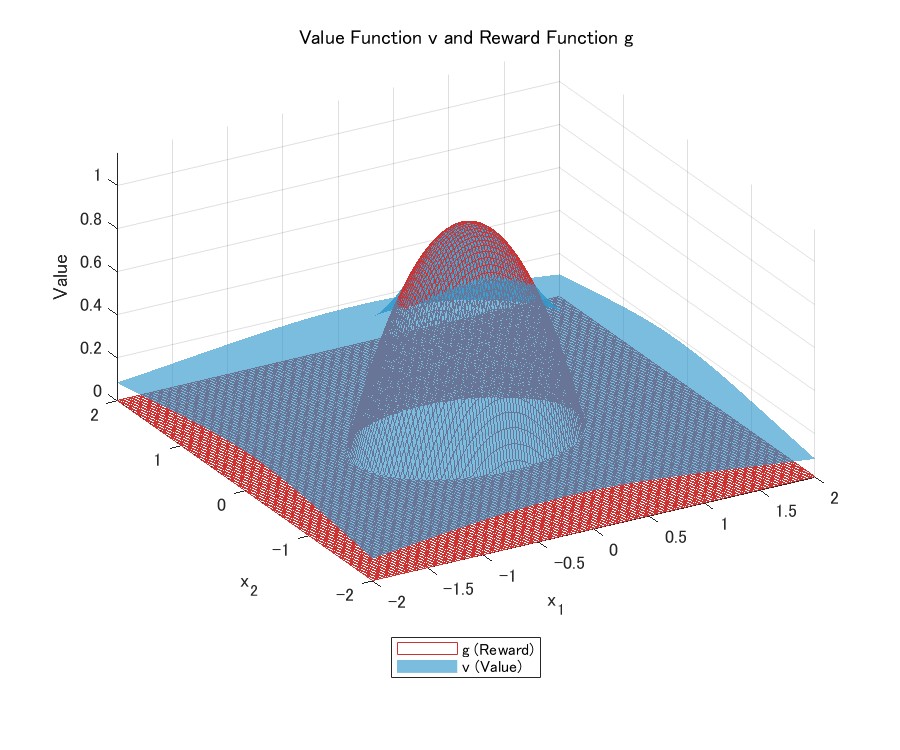}
     \captionsetup{font=scriptsize}
    \caption{Comparison between the reward function $g$ (red mesh) and the approximated value function $v$ (blue surface).}
  \label{fig:value_function}
    \label{fig:Value_2D}
\end{figure}
\begin{figure}
    \centering
    \includegraphics[width=0.5\linewidth]{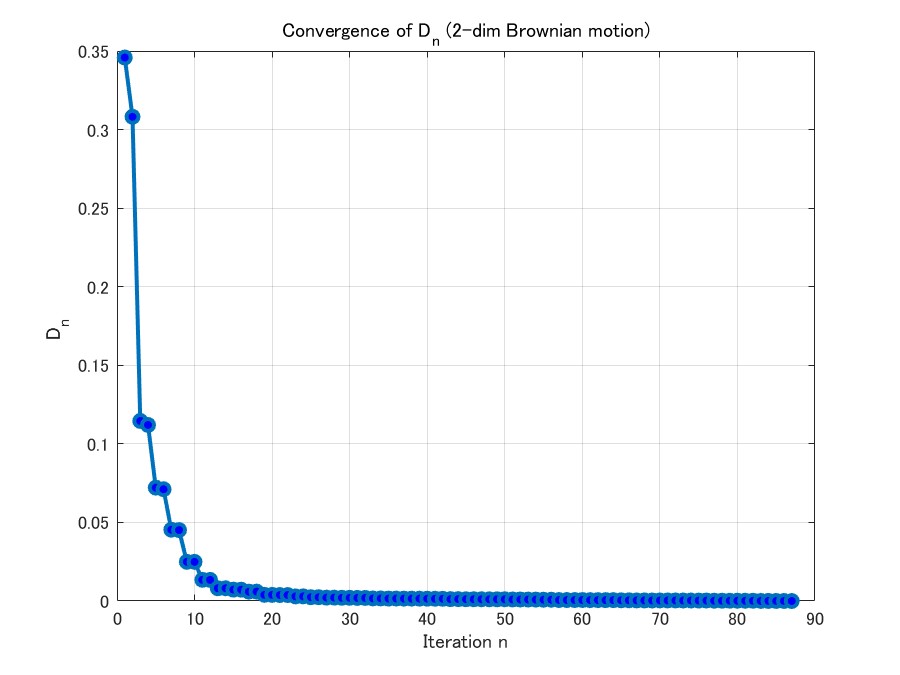}
     \captionsetup{font=scriptsize}
    \caption{Convergence of the duality gap $D_n = D(A_n)$ to zero as a function of the iteration counter $n$.}
    \label{fig:DualityGap_2D}
\end{figure}

\appendix

%\clearpage
\section{Convexity of the stopping region}
\begin{proposition}\label{prop:convex-2D}
The stopping region $\Gamma$ of the optimal stopping in Section \ref{subsec:multi} is convex.
\end{proposition}

\begin{proof}
Let
\[
g_1(x,y):=K-ax^2-by^2;
\]
that is $g(x,y)=(g_1(x,y))_+.$
Denote by $v$ and $v_1$ the value functions
corresponding to $g$ and $g_1$, respectively, and let
$\Gamma$ and $\Gamma_1$ be their stopping regions.
We first show that 
$v=v_1.$
Since $g\geq g_1$, we immediately have
$v\geq v_1.$
On the other hand, we claim that
\[
v_1(x,y)\geq g(x,y)
\qquad\text{for all }(x,y)\in\mathbb R^2.
\]
Indeed, since $\tau=0$ is admissible,
\[
v_1(x,y)\geq g_1(x,y).
\]
If $g_1(x,y)\geq0$, then $g(x,y)=g_1(x,y)$ and hence
\[
v_1(x,y)\geq g_1(x,y)=g(x,y).
\]
If $g_1(x,y)<0$, then $g(x,y)=0$. Moreover, by the recurrence of
the two-dimensional Brownian motion, the process can return to the
region $\{g_1>0\}$, and hence
\[
v_1(x,y)\geq0=g(x,y).
\]
Thus $v_1$ is a $q$-excessive majorant of $g$. Since $v$ is the
smallest $q$-excessive majorant of $g$, it follows that $v\leq v_1.$
Consequently, $v=v_1.$

Next, we show that $\Gamma=\Gamma_1.$
It follows  from $v=v_1$ and $v \leq g$ that $\Gamma_1 \subset \Gamma$.
To prove the reversed inclusion, let $x \in \Gamma$.
Using the recurrence of the two-dimensional Brownian motion again,
$g(x)>0$, and hence $g(x)=g_1(x).$
Therefore, we obtain $g_1(x)=g(x)$ and $x \in \Gamma_1$.
Consequently, $\Gamma=\Gamma_1$.

We next consider the function
\begin{align} \label{eq:qtau}
q^\tau(x,y):=
\mathbb E_{(x,y)}
\left[
e^{-q\tau}g_1(X_\tau)
\right]
-g_1(x,y),
\end{align}

where $\tau$ is an arbitrary stopping time. Put, for $i = 1, 2$,
$$
A_\tau:=\mathbb E[e^{-q\tau}],\qquad
B_\tau^{(i)}
:=\mathbb E[e^{-q\tau}W_\tau^{(i)}],\qquad
C_\tau^{(i)}
:=\mathbb E[e^{-q\tau}(W_\tau^{(i)})^2].
$$

Using the explicit starting state-dependence of $X$, we obtain
\[
\begin{aligned}
q^\tau(x,y)
&=(1-A_\tau)(ax^2+by^2)
-2aB_\tau^{(1)}x
-2bB_\tau^{(2)}y  
+K(A_\tau-1)
-aC_\tau^{(1)}
-bC_\tau^{(2)}.\\
\end{aligned}
\]

Since $0\leq A_\tau\leq1$ by $q>0$ and both $a(1-A_\tau)$ and $b(1-A_\tau)$ are positive,
\[
\left\{(x,y):q^\tau(x,y)\leq0\right\}
\]
is an ellipse (possibly empty or degenerate), and therefore convex.
Note that $(x,y) \in \Gamma_1=\Gamma$ if and only if $q^\tau(x,y) \leq 0$ for all $\tau$.
Hence, we have
$\Gamma = \bigcap_{\tau}
\left\{
(x,y):q^\tau(x,y)\leq0
\right\}.$
Therefore, $\Gamma$ is convex.
\end{proof}
The characterization of the stopping region in terms of \eqref{eq:qtau} follows the argument used in the proof of Proposition 4.1 in \cite{christensen2019optimal}.

%%%%%%%%%%%%%%%%%%%%%%%%%%%%%%%%%%%%%%%%%%%
\bibliographystyle{abbrv}
{\small \bibliography{references}}

\end{document}